\documentclass[11pt,reqno]{amsart}
\usepackage{fonttable}
\usepackage{graphicx,color}
\usepackage{amssymb,amscd,latexsym, mathrsfs, epsfig}
\usepackage[all]{xy}
\usepackage[latin1]{inputenc}
\usepackage{extarrows}
\usepackage{geometry}
\usepackage{geometry}
\newtheorem{thm}{Theorem}[section]
\newtheorem{lem}[thm]{Lemma}
\newtheorem{pro}[thm]{Proposition}
\theoremstyle{definition}
\newtheorem{rem}[thm]{Remark}
\newtheorem{ex}[thm]{Example}

\theoremstyle{plain}

\newtheorem{kor}[thm]{Corollary}
\newtheorem{defi}[thm]{Definition}

\newtheorem{que}[thm]{Question}

\newcommand{\filt}[6]{
\[
\begin{xy}
\xymatrix@R20pt@C20pt{
&\mathbb{C}^3&\\\langle #1,#2 \rangle\ar[ru]&\langle #3,#4\rangle\ar[u]&\langle #5,#6\rangle\ar[lu]\\\langle #1\rangle\ar[u]&\langle #3\rangle\ar[u]&\langle #5\rangle\ar[u]}
\end{xy}
\]
}
\newcommand{\quiverext}[6]{\[
\begin{xy}
\xymatrix@R20pt@C20pt{
(#1)&&(#2)\ar[ld]^{#5}\ar[ll]^{#4}\\&(#3)\ar[lu]^{#6}}
\end{xy}
\]
}
\newcommand{\quiverhom}[6]{\[
\begin{xy}
\xymatrix@R20pt@C20pt{
(#1)&&(#2)\ar@{--}[ld]^{#5}\ar[ll]^{#4}\\&(#3)\ar[lu]^{#6}}
\end{xy}
\]
}

\newcommand\Zn{\mathbb{Z}}
\newcommand\Nn{\mathbb{N}}

\newcommand{\Sc}[2]{\langle #1,#2\rangle}
\newcommand{\Hom}{\mathrm{Hom}} 
\newcommand{\Ext}{\mathrm{Ext}} 
\newcommand{\ext}{\mathrm{ext}} 
\newcommand{\End}{\mathrm{End}} 
\newcommand{\Rep}{\mathrm{Rep}} 
\newcommand{\ses}[3]{0\rightarrow #1\rightarrow #2\rightarrow#3\rightarrow 0}
\newcommand{\sesv}[5]{0\rightarrow #1\xrightarrow{#4} #2\xrightarrow{#5}#3\rightarrow 0}

\DeclareSymbolFont{symbolsC}{U}{pxsyc}{m}{n}
\DeclareMathSymbol{\Perp}{\mathrel}{symbolsC}{121}

\begin{document}
\title{On the recursive construction of indecomposable quiver representations}
\begin{abstract}
For a fixed root of a quiver, it is a very hard problem to construct all or even only one indecomposable representation with this root as dimension vector. We investigate two methods which can be used for this purpose. In both cases we get an embedding of the category of representations of a new quiver
into the category of representations of the original one which increases dimension vectors. Thus it can be used to construct indecomposable representations of the original quiver recursively. Actually, it turns out that there is a huge class of representations which can be constructed using these methods.
\end{abstract}
\author{Thorsten Weist}
\address{Thorsten Weist\\Fachbereich C - Mathematik\\
Bergische Universität Wuppertal\\
D - 42097 Wuppertal, Germany}
\email{weist@uni-wuppertal.de}
\maketitle

\section{Introduction}
\noindent For a fixed quiver, Kac proved in \cite{kac} that the set of dimension vectors corresponding to the set of indecomposable representations coincides with the root system of the underlying graph of the quiver. Thus since the root system is independent of the orientation of the arrows, the question of the existence of indecomposables with this root as dimension vector is also independent. Actually, he also proved that there exists a certain parameter family of indecomposables where the number of parameters depends on the Euler form of the quiver. For a fixed root this raises the question which of the representations in the variety of representations are indecomposable. In general, this is very hard to decide. But in many cases it is possible to construct indecomposable representations recursively. Mostly the idea is to fix a set of indecomposables of smaller dimensions and to glue them appropriately in order to get indecomposable representations of greater dimension. In most cases the new indecomposables can be seen as the middle terms of certain exact sequences. Examples for this are Ringel's reflection functor, see \cite[Section 1]{rin}, and Schofield induction, see \cite[Section 2]{sch}. Also the methods of Peternell presented in \cite[Section 4]{pet} and of the author in \cite[Section 3]{wei} state a recursive construction. Even if it is rarely possible to construct all indecomposable representations of a fixed dimension with these methods, a first step would be to construct an indecomposable tree module for every root. Its existence is conjectured in \cite{rin3} and proved in several cases. In fact, all of the mentioned constructions can be used to construct tree modules. To do so we have to assure that the smaller indecomposables are tree modules and that we choose tree-shaped bases of the groups of extensions between the respective representations. The existence of indecomposable tree modules is known for exceptional roots, see \cite{rin1}, for imaginary Schur roots, see \cite{wei}, and for a lot of cases which are mostly obtained by the above mentioned constructions, see \cite{wie} as an example. 

The starting point of our investigations is to fix a sequence of representations of a fixed quiver. Then we can consider the quiver where the vertices are in one-to-one correspondence with the representations and where the number of arrows is the dimension of the respective groups of extensions. This gives a faithful functor from the representation category of the new quiver to the representation category of the original one. In general, it is not clear which additional properties this functor has. In most cases, it does not preserve morphism spaces or even indecomposability. Since Schofield induction and also parts of Ringel's reflection functor can be realised using this functor, this leads to the following question:
\begin{que}
Which conditions do such sequences of quiver representations have to satisfy if the functor is meant to preserve indecomposability?
\end{que}
Dealing with this question is one of the main issues of this paper. We consider two cases of conditions on such sequences which make sure that the functor is either a fully faithful embedding or at least preserves indecomposability. 

The first incident of sequences fills the gap to quivers with oriented cycles which are not covered by \cite{sch}. More detailed, we fix a sequence of (possibly exceptional) Schur representations such that there are pairwise no homomorphisms between them. In particular, we allow extensions in both directions. Thus, in comparison to the exceptional sequences appearing in Schofield induction and which  are a special case of the sequences under consideration, the induced quiver may have oriented cycles. Thus it gives a huge number of examples which are not covered by it. Moreover, we will see that even in the cases which are also covered by other functors or methods, it is often easier to control the recursive construction in this way.

The second condition, which we investigate, leads to a generalization of some of the functors considered by Ringel in \cite{rin}. In there, initially, an exceptional representation is fixed. Then different categories as the induced perpendicular categories are considered. For instance, for a fixed indecomposable representation which has neither homomorphisms nor extensions to the exceptional representation this functor can be used to construct new indecomposable representations by gluing the exceptional representation a prescribed number of times. This can be realized by a functor from the category of representations of the quiver which has two vertices and $d$ arrows in one direction to the category of representations of the original quiver. Here $d$ denotes the dimension of the extension space. We generalize this to the effect that we allow certain exceptional sequences of representations with pairwise vanishing homomorphism spaces instead of only one exceptional representation.
 This again gives a huge number of new examples.

A natural and very difficult question is for which roots at least one of the constructions mentioned in the introduction applies or to be precise:
\begin{que}\label{1}
For which non-Schurian roots of a fixed quiver $Q$ there exists at least one indecomposable representation which can be constructed using the main results of this paper, i.e. Theorems \ref{thm1} and \ref{indecompconstr} (which include Ringel's reflection functor and Schofield induction)?
\end{que}
Actually it seems that there is a positive answer for the ''large majority'' of roots. Even a root for which there is no indecomposable which can be constructed using the main results of the paper is not known to me. Here we can restrict to non-Schurian roots because Schurian roots are covered by \cite{wei}. The methods presented there can also be seen as a recursive application of the functor under consideration. We should mention that it is not straightforward to compare the results of Peternell with the preceding results because the resulting quivers have different properties, see also the comment in the beginning of Section 3.1.

Actually, Question \ref{1}, which is also considered in the fourth section, is closely related to the still open problem asking for the existence of indecomposable tree modules. Indeed, if there were a positive answer for every root, we could get a recursion started proving that for every root there exists an indecomposable tree module.

In the last section we deal with the natural question whether it is possible to take self-extensions of non-exceptional representations into account when considering the described functor. This should translate to adding loops to a vertex which corresponds to a Schur representation with self-extensions. Actually, there is a natural way of generalization, but we can state a counterexample which shows that, in general, this functor does not preserve indecomposability even in the case of a quiver with only one vertex and loops. 
\vspace{0.3cm}

\noindent {\bf Acknowledgements:} I would like to thank K. Bongartz for valuable discussions on this topic.
Moreover, I would like to thank R. Kinser, M. Reineke and C.M. Ringel for some helpful comments.

\section{Recollection and Notation}\label{allg}
\subsection{Representations of quivers}\label{repsec}
\noindent Let $k$ be an algebraically closed field. Let $Q=(Q_0,Q_1)$ be a quiver with vertices $Q_0$ and arrows $Q_1$ denoted by $\rho:q\rightarrow q'$ for $q,q'\in Q_0$. A vertex $q\in Q_0$ is called sink if there does not exist an arrow $\rho:q\rightarrow q'\in Q_1$. A vertex $q\in Q_0$ is called source if there does not exist an arrow $\rho:q'\rightarrow q\in Q_1$. 
Except for Section \ref{5} we only consider quivers without loops, i.e. arrows $\rho:q\to q$.
 
Define the abelian group
\[\mathbb{Z}Q_0=\bigoplus_{q\in Q_0}\mathbb{Z}q\] and its monoid of dimension vectors $\mathbb{N}Q_0$.

A finite-dimensional $k$-representation of $Q$ is given by a tuple
\[X=((X_q)_{q\in Q_0},(X_{\rho}:X_q\rightarrow X_{q'})_{\rho:q\to q'\in Q_1})\]
of finite-dimensional $k$-vector spaces and $k$-linear maps between them. The dimension vector $\underline{\dim}X\in\mathbb{N}Q_0$ of $X$ is defined by
$\underline{\dim}X=\sum_{q\in Q_0}\dim_kX_qq$. In the following, we mostly write $\dim$ instead of $\dim_k$. We denote by $\mathrm{Rep}(Q)$ the category of finite-dimensional representations of $Q$. By $S_q$ we denote the simple representation corresponding to the vertex $q$.

Let $\alpha\in\mathbb{N}Q_0$ be a dimension vector. The variety $R_\alpha(Q)$ of $k$-representations of $Q$ of dimension vector
$\alpha$ is defined as the affine $k$-space
\[R_{\alpha}(Q)=\bigoplus_{\rho:q\rightarrow q'\in Q_1} \mathrm{Hom}_k(k^{\alpha_q},k^{\alpha_{q'}}).\]

On $\Zn Q_0$ we have a non-symmetric bilinear form, the Euler form,
which is defined by
\[\Sc{\alpha}{\beta}=\sum_{q\in Q_0}\alpha_q\beta_q-\sum_{\rho:q\rightarrow q'\in Q_1}\alpha_q\beta_{q'}\]
for $\alpha,\,\beta\in\Zn Q_0$.

For more details concerning the connection between roots and indecomposable representations we refer to \cite[Section 2]{kac}. We only give a short summary. A dimension vector is called a root if there exists an indecomposable representation of this dimension. It is called Schur root if there exists a representation with trivial endomorphism ring. By $W(Q)$ we denote the Weyl group of the quiver $Q$. A root $\alpha\in\Nn Q_0$ is called real if we have $\alpha\in W(Q)Q_0$, i.e. $\alpha$ arises by reflecting a simple root. All the other roots are called imaginary. Recall that a root is real if and only if $\Sc{\alpha}{\alpha}=1$ and imaginary if and only if $\Sc{\alpha}{\alpha}\leq 0$. In the following we do not always distinguish between a real root and the unique indecomposable representation of this dimension. If $\alpha$ is a real root of a quiver, we will denote the unique indecomposable representation of this dimension by $M_{\alpha}$. 

Let $(\alpha,\beta):=\Sc{\alpha}{\beta}+\Sc{\beta}{\alpha}$ be the symmetrized Euler form. The fundamental domain $F(Q)$ of $\mathbb{N}Q_0$ is given by the dimension vectors $\alpha$ with connected support such that $(\alpha,q)\leq 0$ for all $q\in Q_0$. Moreover, we have $\alpha\in W(Q)F(Q)$ for all imaginary roots $\alpha$.

For the remaining part of this section we assume that $Q$ has no oriented cycles.
Let $X$ and $Y$ be two representations of a quiver $Q$. We consider the linear map
\[d_{X,Y}:\bigoplus_{q\in Q_0}\Hom_k(X_q,Y_q)\rightarrow\bigoplus_{\rho:q\rightarrow q'\in Q_1}\Hom_k(X_q,Y_{q'})\]
defined by $d_{X,Y}((f_q)_{q\in Q_0})=(Y_{\rho}f_q-f_{q'}X_{\rho})_{\rho:q\rightarrow q'\in Q_1}$.
Then we have $\ker(d_{X,Y})=\Hom_Q(X,Y)$ and $\mathrm{coker}(d_{X,Y})=\Ext_Q(X,Y)$, see \cite[Section 2.1]{rin2}. In the following, we mostly drop the subscript $Q$.

Recall that every morphism $g\in\bigoplus_{\rho:q\rightarrow q'\in Q_1}\Hom_k(X_q,Y_{q'})$ defines an exact sequence $E(g)\in\Ext(X,Y)$ by
\[\ses{Y}{((Y_q\oplus X_q)_{q\in Q_0},(\begin{pmatrix}Y_{\rho}&g_{\rho}\\0&X_{\rho}\end{pmatrix})_{\rho\in Q_1})}{X}\]
with the canonical inclusion on the left hand side and the canonical projection on the right hand side. Then it is straightforward to check that two sequences $E(g)$ and $E(h)$ are equivalent if and only if $g-h\in\mathrm{Im}(d_{X,Y})$.

Moreover, by \cite[Section 2.2]{rin2}, for two representations $X$, $Y$ of $Q$ we have
\[\Sc{\underline{\dim} X}{\underline{\dim} Y}=\dim\Hom(X,Y)-\dim\Ext(X,Y)\]
and $\Ext^i(X,Y)=0$ for $i\geq 2$. 

If some property is independent of the point chosen in some open subset $U$ of $R_{\alpha}(Q)$, following \cite{sch}, we say that this property is true for a general representation of dimension vector $\alpha\in\Nn Q_0$.

Since the function
$\lambda:R_{\alpha}(Q)\times R_{\beta}(Q)\rightarrow\mathbb{N},\,(X,Y)\mapsto\dim \Hom(X,Y)$, is upper semi-continuous, see for instance \cite[Section 1]{sch}, we can define $\hom(\alpha,\beta)$ to be the minimal, and therefore general, value of this function. In particular, we get that if $\alpha$ is a Schur root of a quiver, then a general representation is Schurian. Moreover, we define $\ext(\alpha,\beta):=\hom(\alpha,\beta)-\Sc{\alpha}{\beta}$.

\subsection{A functor between categories of representations of quivers}\label{functorsec}
The main focus of this paper is on a functor between two categories of quiver representations. This functor can be defined for every fixed sequence of representations of a fixed quiver $Q$. But, in general, this functor does not have any nice properties as it does not preserve indecomposability or homomorphism spaces. Thus the main aim is to investigate under which extra conditions the functor preserves indecomposability. 

To define the functor we fix a sequence $M=(M_1,\ldots,M_r)$ consisting of representations of a quiver $Q$. The next step is to consider the quiver $Q(M)$ which has vertices $Q(M)_0=\{m_1,\ldots,m_r\}$. Moreover, the quiver has $n_{ij}:=\dim\Ext(M_i,M_j)$ arrows from $m_i$ to $m_j$ if $i\neq j$ and no loops. For each pair $i,j$ with $i\neq j$ we also fix $$\mathcal{B}_{ij}=\{\chi^{ij}_1,\ldots,\chi^{ij}_{n_{ij}}\}\subseteq\bigoplus_{\rho:q\rightarrow q'\in Q_1}\Hom_k((M_i)_q,(M_j)_{q'})$$ such that the corresponding residue classes are a basis of $\Ext(M_i,M_j)$. Since the arrows of $Q(M)$ are in correspondence with these basis elements we denote the arrows of $Q(M)$ by $\chi_{l}^{ij}$. Finally, for a representation $X$ of $Q(M)$ define $\tilde X_{i,q}:=(M_i)_q\otimes_k X_{m_i}$ where $q\in Q_0$ and $i\in\{1,\ldots,r\}$. 

This gives rise to a functor $F_M:\Rep(Q(M))\rightarrow \Rep(Q)$: we define a representation $F_MX$ of $Q$ by the vector spaces
\[(F_MX)_q=\bigoplus_{i=1}^r \tilde X_{i,q}\text{ for all }q\in Q_0\]
and for $\rho:q\rightarrow q'$ we define linear maps $(F_MX)_{\rho}=\bigoplus_{i=1}^r \tilde X_{i,q}\to\bigoplus_{i=1}^r \tilde X_{i,q'}$ by
\[((F_MX)_{\rho})_{i,i}=(M_i)_{\rho}\otimes_k\mathrm{id}_{X_{m_i}}:\tilde X_{i,q}\rightarrow \tilde X_{i,q'}\]
and 
\[((F_MX)_{\rho})_{i,j}=\sum_{l=1}^{n_{ji}}(\chi^{ji}_l)_{\rho}\otimes_kX_{\chi^{ji}_l}:\tilde X_{j,q}\rightarrow \tilde X_{i,q'}\]
for $i\neq j$. 

Let $f=(f_{m_i})_{i=1,\ldots,r}:X\rightarrow X'$ be a morphism. Then we define $F_Mf:F_MX\rightarrow F_MX'$ by
\[((F_Mf)_q)_{i,j}=\left\{\begin{matrix}\mathrm{id}_{(M_j)_q}\otimes_kf_{m_j}:\tilde X_{j,q}\rightarrow \tilde X'_{i,q}\text{ if }i=j\\0:\tilde X_{j,q}\rightarrow \tilde X'_{i,q}\text{ if }i\neq j\end{matrix}\right..\]

In abuse of notation we will often skip the $M$ in $F_M$. Note that $F$ indeed defines a functor because for a morphism $f:X\rightarrow X'$ we have that 
\begin{eqnarray*}
((FX')_{\rho}\circ (Ff)_q)_{i,j}&=&\sum_{l=1}^r((FX')_{\rho})_{i,l}\circ((Ff)_q)_{l,j}=((FX')_{\rho})_{i,j}\circ((Ff)_q)_{j,j}\\
&=&\left(\sum_{l=1}^{n_{ji}}(\chi^{ji}_l)_{\rho}\otimes_kX'_{\chi^{ji}_l}\right)\circ\mathrm{id}_{(M_j)_q}\otimes_kf_{m_j}\\
&=&\mathrm{id}_{(M_i)_{q'}}\otimes_kf_{m_i}\circ\left(\sum_{l=1}^{n_{ji}}(\chi^{ji}_l)_{\rho}\otimes_kX_{\chi^{ji}_l}\right)\\
&=&((Ff)_{q'})\circ (FX)_{\rho})_{i,j}.
\end{eqnarray*} 
Thus it follows that $(FX')_{\rho}\circ (Ff)_q=(Ff)_{q'}\circ (FX)_{\rho}$ for all $\rho:q\rightarrow q'$.

\begin{rem}\label{bem}
Note that it is straightforward to check that the functor $F$ is always faithful.
Moreover, the definition of the functor can be summarized as follows: every vertex of $Q(M)$ corresponds to a representation of $M$. Moreover, every arrow of $Q(M)$ corresponds to a basis element of the group of extensions of the representations corresponding to the tail and the head of the arrow. Taking one copy of $M_i$ and $M_j$ a basis element of $\Ext(M_i,M_j)$ can be thought of as a way of gluing $M_i$ to $M_j$. In order to obtain the representation $FX$ for a fixed representation $X$ of dimension $\alpha$ we first take $\alpha_i$ copies of $M_i$ for every $i\in\{1,\ldots,r\}$. Finally, the linear maps of $X$ describe how to glue the copies of $M_i$ and $M_j$.
\end{rem}
\begin{rem} Note that, if all representations $M_i$ are tree modules, the Ext-bases are tree shaped and the representation $X$ is also a tree module, this gluing can be made very explicit as we will see in many examples throughout the paper. Roughly speaking, in this case gluing translates to drawing extra arrows between the corresponding coefficient quivers of the different tree modules.
We review the definitions of tree modules and tree-shaped bases respectively in Section \ref{tree}.
\end{rem}

If $X$ is a representation of $Q(M)$, by $S_iX$ we denote the representation defined by \[(S_iX)_{m_j}=\left\{\begin{matrix}X_{m_i} \text{ if } j=i\\0\text{ otherwise}\end{matrix}\right..\]
Note that $FS_iX=M_i^{\dim X_{m_i}}$. Therefore, for $i\neq j$ there exist isomorphisms
\[\Theta^{ji}_{X,X'}:\bigoplus_{l=1}^{n_{ji}}\Hom_k(S_jX,S_iX')\cong\Ext(M_j^{\dim X_{m_j}},M_i^{\dim X'_{m_i}})\]
induced by $(\phi_l)_{l}\mapsto (\sum_{l=1}^{n_{ji}}(\chi^{ji}_l)_{\rho}\otimes_k\phi_l)_{\rho\in Q_1}=F((\phi_l)_l)$. It is straightforward to check that $\Theta^{ji}$ is functorial in $X$ and $X'$. This means that, for a morphism $f=(f_i)_{i=1,\ldots,r}:X'\to Y$ with $X,\,Y,\,X'\in \Rep(Q(M))$, we have $\Theta_{X,X'}^{ji}(f_i\circ(\phi_l)_l)=F(f_i\circ(\phi_l)_l)= F(f)_i\circ F((\phi_l)_l)$. We obtain the analogue statement for morphisms $f:X\to Y$.

\subsection{Exceptional sequences and perpendicular categories}\label{experp}
An indecomposable representation $X$ of a quiver $Q$ is called exceptional if $\Ext(X,X)=0$. Then it follows that $\underline{\dim} X$ is a real Schur root and $\mathrm{End}(X)=k$, see \cite[Lemma 4.1]{hr}. A sequence $E=(E_1,\ldots,E_r)$ of representations of $Q$ is called exceptional if every $E_i$ is exceptional and, moreover, $\Hom(E_i,E_j)=\Ext(E_i,E_j)=0$ if $i<j$. If, in addition, $\Hom(E_j,E_i)=0$ if $i<j$, we call such a sequence reduced. For two roots $\alpha$ and $\beta$ we denote by $\beta\in\alpha^{\perp}$ if $\hom(\alpha,\beta)=\ext(\alpha,\beta)=0$. In this way we can also refer to exceptional sequences of roots.

For a set $M=\{M_1,\ldots,M_r\}$ of representations of $Q$ we define its perpendicular categories
\[^{\perp}M=\{X\in\mathrm{Rep}(Q)\mid \Hom(X,M_j)=\Ext(X,M_j)=0\text{ for }j=1,\ldots,r\},\] 
\[M^{\perp}=\{X\in\mathrm{Rep}(Q)\mid \Hom(M_j,X)=\Ext(M_j,X)=0\text{ for }j=1,\ldots,r\}.\]

It is straightforward to check that these categories are closed under direct sums, direct summands, extensions, images, kernels and cokernels. From \cite[Theorems 2.3 and 2.4]{sc2} it follows:
\begin{thm}\label{perpcat}Let $Q$ be a quiver with $n$ vertices and $E=(E_1,\ldots,E_r)$ be an exceptional sequence. 
\begin{enumerate}
\item The categories $^{\perp}E$ and $E^{\perp}$ are equivalent to the categories of representations of quivers $Q(^{\perp}E)$ and $Q(E^{\perp})$ respectively such that these quivers have $n-r$ vertices and no oriented cycles.
\item There is an isometry with respect to the Euler form between the dimension vectors of  $Q(^{\perp}E)$ (resp. $Q(E^{\perp})$) and the dimension vectors of $^{\perp}E$ (resp. $E^{\perp}$) given by
$\Phi((d_1,\ldots,d_{n-r}))=\sum_{i=1}^{n-r}d_i\alpha_i$ where $\alpha_1,\ldots,\alpha_{n-r}$ are the dimension vectors of the simple representations of the perpendicular categories.
\end{enumerate}
\end{thm}
For an exceptional sequence $E=(E_1,\ldots,E_r)$ with $\underline{\dim} E_i=\alpha_i$ we define $\mathcal{C}(E_1,\ldots,E_r)$ as the full subcategory of $\Rep(Q)$ which contains $E_1,\ldots,E_r$ and which is closed under extensions, kernels of epimorphisms and cokernels of monomorphisms. Suppose that $E=(E_1,\ldots,E_r)$ is a reduced exceptional sequence. Then by \cite[Lemma 2.35]{dw2} we have that $E_1,\ldots,E_r$ are the simple objects of $\mathcal{C}(E_1,\ldots,E_r)$. Moreover, by Theorem \ref{perpcat}, it follows that the category $\mathcal{C}(E_1,\ldots,E_r)$ is equivalent to the category of representations of the quiver $Q(E)$ defined in Section 
\ref{functorsec}.
Thus an immediate consequence of Theorem \ref{perpcat} is the following, see also \cite[Section 2]{sch} and \cite[Theorem 2.38]{dw}:
\begin{kor}\label{root}
Let $E=(E_1,\ldots,E_r)$ with $\underline{\dim} E_i=\alpha_i$ be a reduced exceptional sequence. Then $\alpha=\sum_{i=1}^rd_i\alpha_i$ is a root of $Q$ if and only if $(d_1,\ldots,d_r)$ is a root of $Q(E)$.
\end{kor}

\subsection{Ringel's reflection functor}
We review some of the results of \cite[Section 1]{rin}. Fixing an exceptional representation $S$ and a full subcategory $\mathcal{C}$ of $\Rep(Q)$ we denote by $\mathcal{C}/S$ the category which has the same objects as $\mathcal{C}$ and the same maps modulo those factorizing through $\bigoplus^n_{i=1} S$ for some $n\in\Nn$. We define the following full subcategories of $\Rep(Q)$:
\begin{enumerate}
\item $\mathcal{M}^{-S}=\{X\in\Rep(Q)\mid\Hom(X,S)=0\}$ 
\item $\mathcal{M}_{-S}=\{X\in\Rep(Q)\mid \Hom(S,X)=0\}$.
\item $\mathcal{M}^S$ as the category of representations $X\in\Rep(Q)$ with $\Ext(S,X)=0$ such that, moreover, there does not exist a direct summand of $X$ which can be embedded into a direct sum of copies of $S$. 
\item $\mathcal{M}_S$ as the category of representations $X\in\Rep(Q)$ with $\Ext(X,S)=0$ such that, moreover, no direct summand of $X$ is a quotient of a direct sum of copies of $S$.

\end{enumerate}

Let $X\in\mathcal{M}^S$ and $\mathcal{B}:=\{\varphi_1,\ldots,\varphi_n\}$ be a basis of $\Hom(X,S)$. Following \cite[Lemma 2]{rin}, there exists an exact sequence 
\[\ses{X^{-S}}{X}{\bigoplus_{i=1}^n S}\]
induced by $\mathcal{B}$ such that the induced sequences $e_1,\ldots,e_n$ form a basis of $\Ext(S,X^{-S})$. Moreover, we have $X^{-S}\in\mathcal{M}^{-S}$. The other way around, if $Y\in\mathcal{M}^{-S}$ and $\{e_1,\ldots,e_n\}$ is a basis of $\Ext(S,Y)$ we have an induced sequence
\[\ses{Y}{Y^S}{\bigoplus_{i=1}^nS}\]
such that $Y^{S}\in\mathcal{M}^S$. We can proceed similarly for $X\in\mathcal M_S$ and $Y\in\mathcal M_{-S}$. Then we have the following theorem summarizing the results of \cite[Section 1]{rin}:
\begin{thm}\label{ringel}
\begin{enumerate} 
\item  There exists an equivalence of categories given by the functor $F:\mathcal{M}^S/S\rightarrow\mathcal{M}^{-S},\,X\mapsto X^{-S}.$
\item There exists an equivalence of categories given by the functor $G:\mathcal{M}_S/S\rightarrow\mathcal{M}_{-S},$ $X\mapsto X_{-S}$.
\item There exist equivalences $\Psi:\mathcal{M}^{S}_{-S}\rightarrow\mathcal{M}^{-S}_{S}$ and $\Phi:\mathcal{M}^S_S/S\rightarrow\mathcal{M}^{-S}_{-S}$ induced by composing the functors from above.
\end{enumerate}
\end{thm}
\subsection{Tree and tree-shaped bases}\label{tree}
We introduce coefficient quivers and tree modules following \cite{rin1}, see also \cite{wei}. Let $X$ with $\underline{\dim} X=\alpha$ be a representation of $Q$. A basis of $X$ is a subset $\mathcal{B}$ of $\bigoplus_{q\in Q_0}X_q$ such that
\[\mathcal{B}_q:=\mathcal{B}\cap X_q\] is a basis of $X_q$ for all vertices $q\in Q_0$. For every arrow $\rho:q\rightarrow q'$ we may write $X_{\rho}$ as a $(\alpha_{q'}\times \alpha_q)$-matrix $X_{\rho,\mathcal{B}}$ with coefficients in $k$ such that the rows and columns are indexed by $\mathcal{B}_{q'}$ and $\mathcal{B}_q$ respectively.
\begin{defi}
The coefficient quiver $\Gamma(X,\mathcal{B})$ of a representation $X$ with a fixed basis $\mathcal{B}$ has vertex set $\mathcal{B}$ and arrows between vertices are defined by the condition: if $(X_{\rho,\mathcal{B}})_{b',b}\neq 0$, there exists an arrow $(\rho,b,b'):b\rightarrow b'$ where $b\in \mathcal B_q$, $b'\in \mathcal B_{q'}$ and $\rho:q\to q'$.\\
A representation $X$ is called a tree module if there exists a basis $\mathcal{B}$ for $X$ such that the corresponding coefficient quiver is a tree.
\end{defi}

Let $M_{m,n}(k)$ be the set of $m\times n$ matrices with coefficients in $k$ and for $A\in M_{m,n}(k)$ denote by $A_{i,j}$ the $(i,j)$-entry. We denote by $E(s,t)\in M_{m,n}(k)$ the matrix with $E(s,t)_{s,t}=1$ and zero otherwise. If $Y$ is another representation of $Q$, we call a basis $\{E(f_1),\ldots,E(f_n)\}$ of $\Ext(X,Y)$ tree-shaped if for all $i=1,\ldots, n$ there exists $s,t$ such that we have $(f_i)_{\rho}=E(s,t)$ for exactly one $\rho\in Q_1$ and $(f_i)_{\rho '}=0$ for $\rho '\neq \rho$. Since we can clearly choose a tree-shaped basis $\mathcal{B}$ of $\bigoplus_{\rho:q\rightarrow q'\in Q_1}\Hom_k(X_q,Y_{q'})$, we can choose a basis of $\Ext(X,Y)$ consisting of elements of the form $b+\mathrm{Im}(d_{X,Y})$ with $b\in\mathcal{B}$, see also the proof of \cite[Lemma 3.16]{wie}.

In order to construct a tree module and its coefficient quiver respectively, it is often useful to consider the universal covering quiver $\tilde Q$ of the given quiver $Q$. To do so, we use the notation of \cite[Section 2]{wei}.

Let $W_Q=\langle \rho,\rho^{-1}\mid \rho\in Q_1\rangle$ be the set of words in $Q_1$. Every representation $\tilde{X}$ of $\tilde{Q}$ gives rise to a representation of $X$ of $Q$ in the following way:
\[X_q=\bigoplus_{w\in W_Q} \tilde{X}_{(q,w)},\quad q\in Q_0\]
and for an arrow $\rho:q\to q'$ we define $X_{\rho}:X_q\rightarrow X_{q'}$ by $X_{\rho}|_{\tilde X_{(q,w)}}=\tilde{X}_{\rho_{(q,w)}}$. Now we can make use of the following result, see \cite[Lemma 3.5]{gab}:
\begin{thm}
If $\tilde{X}$ is an indecomposable representation of $\tilde{Q}$, the corresponding representation $X$ of $Q$ is also indecomposable.
\end{thm}

Note that it is straightforward to check that every indecomposable tree module is already a representation of a connected component of the universal cover.

For the simple reflection $s_q\in W(Q)$ at $q$ define $\tilde{s}_q:=\prod_{w\in W_Q}s_{(q,w)}\in W(\tilde{Q})$ where $s_{(q,w)}$ is the simple reflection at $(q,w)$. Note that the product can be assumed to be finite when applied to a dimension vector with finite support and, moreover, that $\tilde{s}_{(q,w)}\cdot \tilde{s}_{(q,w')}=\tilde{s}_{(q,w')}\cdot \tilde{s}_{(q,w)}$. Thus $\tilde{s}_q$ is independent of the chosen order. Now we can recursively define $\tilde{\sigma}$ for every $\sigma\in W(Q)$.

Recall the conjecture of Ringel, see \cite[Problem 9]{rin3}, saying that there exists an indecomposable tree module for every root of a quiver without oriented cycles. The following statement assures that it is enough to show the existence of indecomposable tree modules for all roots of quivers which are trees:
\begin{pro}\label{dec}
For every root $\alpha$ of $Q$ there exists a root $\tilde{\alpha}$ of $\tilde{Q}$ such that we have $\sum_{w\in W_Q}\tilde{\alpha}_{(q,w)}=\alpha_q$.
\end{pro}
\begin{proof}
The statement is clear for simple roots. If $\alpha$ is real, then there exists a $\sigma\in W(Q)$ such that $\alpha=\sigma\cdot e_q$ for a simple root $e_q$. Thus $\tilde{\alpha}=\tilde{\sigma}\cdot e_{(q,1)}$, where $1$ denotes the empty word, satisfies the statement. 
Thus let $\alpha$ be an imaginary root. We can assume that $\alpha\in F(Q)$. Indeed, for a general imaginary root we can again apply reflections like we did in the case of real roots.

Every root in the fundamental domain is a Schur root because otherwise there exists a real Schur root $\beta$ in the canonical decomposition such that $s_{\beta}(\alpha)<\alpha$. Thus, since
\[s_{\beta}(\alpha)=\alpha-\sum_{q\in Q_0}\beta_q(\alpha,e_q)\beta<\alpha,\] 
we have $(\alpha,e_q)>0$ for at least one $q\in Q_0$.

By \cite[Theorem 3.19]{wei} there exists an indecomposable tree module for every Schur root and thus for every root in the fundamental domain. But this is, as already mentioned, a representation of the universal covering quiver.
\end{proof}

In the last part of this section, we want to discuss  a question related to Ringel's conjecture. We fix an indecomposable representation $X$ of a quiver $Q$ with $\alpha:=\underline{\dim} X$ and a basis $\mathcal B$ of $X$. Then we denote by $a(X,\mathcal{B})$ the number of arrows of the coefficient quiver $\Gamma(X,\mathcal B)$. There are two natural questions related to the conjecture of Ringel: 
\begin{itemize}
\item What is $m(X):=\min\{ a(X,\mathcal B)\mid \mathcal B\text{ basis of } X\}$? 
\item What is $m(\alpha):=\max\{ m(X)\mid X\text{ indecomposable},\, \underline{\dim}X=\alpha\}?$
\end{itemize}
Since Kac's Theorem predicts a $1-\Sc{\alpha}{\alpha}$-parameter family of indecomposable representation, the first guess could be that we have 
\begin{equation}\label{eq1}\sum_{q\in Q_0}\alpha_q-1\leq m(X)\leq \sum_{q\in Q_0}\alpha_q-\Sc{\alpha}{\alpha},\end{equation}

\begin{equation}\label{eq2}m(\alpha)=\sum_{q\in Q_0}\alpha_q-\Sc{\alpha}{\alpha}.\end{equation} 
Note that the first inequality trivially follows from the indecomposability of $X$. Moreover, $X$ is an indecomposable tree module exactly if it is an equality. In turn for exceptional representations the inequalities clearly hold because they are tree modules by \cite{rin1}. 
Unfortunately, the answer is not always so easy as the following two easy examples show:
\begin{ex}

Consider the Kronecker quiver $Q=K(2)$ and $\alpha=(2,2)$. Up to isomorphism and permutation of the two linear maps, the indecomposables are given by
\[X(\lambda):=((k^2,k^2),(\begin{pmatrix}1&0\\0&1\end{pmatrix},\begin{pmatrix}\lambda&1\\0&\lambda\end{pmatrix}))\]
where $\lambda\in k$. Thus we actually have $m(X(\lambda))=5$ if $\lambda\neq 0$ showing that inequalities (\ref{eq1}) and (\ref{eq2}) are not true in this case.
\end{ex}
\begin{ex}\label{subspace}

 But there also seem to be many examples where the inequalities are true. One of these is obtained when considering the $n$-subspace quiver, i.e. $Q_0=\{q_0,\ldots,q_n\}$ and $Q_1=\{\rho_i:q_i\to q_0\mid i=1,\ldots,n\}$, and the root $(2,1,\ldots,1)$.\end{ex}

This raises the following question: 
\begin{que}
Let $Q$ be a quiver.  
\begin{itemize}
\item For which indecomposable representations $X$ does inequality $(\ref{eq1})$ hold?
\item For which roots does inequality $(\ref{eq2})$ hold?
\end{itemize}
\end{que}

\section{Recursive construction of indecomposable representations}
\noindent In this section we present two conditions on sequences of quiver representations which assure that the functor introduced in Section  \ref{functorsec} preserves indecomposability. Consequently, it can be used to construct indecomposable representations recursively. The first condition fills the gap to quivers with oriented cycles which are not covered by \cite{sch}. Moreover, the second condition which we investigate includes some of the functors considered by Ringel, see Theorem \ref{ringel} and also Remark \ref{remrin}.

\subsection{Indecomposable representations obtained by gluing Schur representations}\label{functor1}
Let $Q$ be quiver without oriented cycles. Recall from Section \ref{experp} that, if we fix a reduced exceptional sequence $E$ in $Q$, the exceptional representations of this sequence are precisely the simple objects of the category $(^\perp E)^{\perp}$. Thus the functor $F_E$ from the category $\Rep (Q(E))$ to $\Rep(Q)$ is a fully faithful embedding. We generalize this to the effect that we do not restrict to reduced exceptional sequences as the building blocks of the induced subcategory of $\Rep (Q)$. We rather allow extensions in both directions, but we also do not allow homomorphisms. Moreover, we do not restrict to exceptional representations. In this case it will turn out that the functor is a fully faithful embedding. Thus it also preserves indecomposable representations. In order to prove this we mainly proceed along the lines of \cite[Section 4]{pet}. In \cite{pet}, W. Peternell restricts to the case where the quiver $Q(M)$ is bipartite, i.e. every vertex is either a sink or a source. In particular, in this case every representation of $Q$ is already given as the middle term of an exact sequence. In return, in \cite{pet} morphisms between certain representations are allowed and, moreover, the representations under consideration do not have to be Schurian. 
 
\begin{defi}
A sequence of Schur roots $S=(\alpha_1,\ldots,\alpha_r)$ such that $\hom(\alpha_i,\alpha_j)=0$ for $i\neq j$ is called an elementary sequence (of Schur roots). 

A sequence of Schur representations $M=(M_1,\ldots,M_r)$ with $\Hom(M_i,M_j)=0$ for $i\neq j$ is called an elementary sequence (of Schur representations).
\end{defi} 
If all roots (resp. representations) are exceptional, we call such a sequence elementary sequence of exceptional roots (resp. exceptional representations).

We fix an elementary sequence $S=(\alpha_1,\ldots,\alpha_r)$. Then we fix an elementary sequence of Schur representations $M=(M_1,\ldots,M_r)$ with $\underline{\dim}M_i=\alpha_i$. Moreover, fix  
$$\mathcal{B}_{ij}=\{\chi^{ij}_1,\ldots,\chi^{ij}_{n_{ij}}\}\subseteq\bigoplus_{\rho:q\rightarrow q'\in Q_1}\Hom_k((M_i)_q,(M_j)_{q'})$$ such that the corresponding residue classes are a basis of $\Ext(M_i,M_j)$. Recall that we can always assume that this basis is tree-shaped. This is very helpful in order to construct indecomposable tree modules.
We denote by $\mathrm{add} (M_i)$ the full subcategory of $\mathrm{Rep}(Q)$ whose objects consist of direct sums of $M_i$. Denoting by $F=F_M$ the functor introduced in Section \ref{functorsec} and using the notation introduced there, we have the following lemma:
\begin{lem}\label{hom0}
\begin{enumerate}
\item The functor $F$ induces a fully faithful embedding $$F|_{M_i}:\Rep(Q(M_i))\rightarrow\mathrm{add}(M_i).$$
\item Let $X$ and $X'$ be two representations of $Q(M)$. Let $r=2$ and assume that $X_{\chi_l^{12}}=X'_{\chi_l^{12}}=0$ for $l=1,\ldots,n_{12}$. Then for every morphism $f:FX\rightarrow FX'$ we have
\[f=(f_q)_{q\in Q_0}=\begin{pmatrix}(f_q)_{1,1}&0\\0&(f_q)_{2,2}
\end{pmatrix}_{q\in Q_0}\]
with $f_{ii}:=((f_q)_{i,i})_{q\in Q_0}\in\Hom_{Q}(FS_iX,FS_iX')$.
\end{enumerate}
\end{lem}
\begin{proof}
The first statement follows because we have $\Hom_Q(M_i,M_i)=k$.

Let $f:FX\rightarrow FX'$ be a morphism. Thus we have $(FX')_{\rho}\circ f_q=f_{q'}\circ (FX)_{\rho}$ for every $\rho:q\to q'$. Since $X_{\chi_l^{12}}=X'_{\chi_l^{12}}=0$ for $l=1,\ldots,n_{12}$ by assumption, we have $((FX)_{\rho})_{2,1}=0$ and $((FX')_{\rho})_{2,1}=0$. Using $\Hom_Q(M_1,M_2)=0$ we obtain that $f$ is of the form
\[f=(f_q)_{q\in Q_0}=\begin{pmatrix}(f_q)_{1,1}&(f_q)_{1,2}\\0&(f_q)_{2,2}
\end{pmatrix}_{q\in Q_0}\]
with morphisms $f_{i,i}\in\Hom_Q(FS_iX,FS_iX')$ and a linear map $$f_{1,2}:=((f_q)_{1,2})_{q\in Q_0}\in\bigoplus_{q\in Q_0}\Hom_k(\tilde X_{2,q},\tilde X'_{1,q}).$$ Moreover, we get
\begin{eqnarray*}d_{FS_2X,FS_1X'}(f_{1,2})&=&(((FX')_{\rho})_{1,1}\circ(f_{q})_{1,2}-(f_{q'})_{1,2}\circ ((FX)_{\rho})_{2,2}))_{\rho:q\to q'\in Q_1}\\
&=&((f_{q'})_{1,1}\circ ((FX)_{\rho})_{1,2}-((FX')_{\rho})_{1,2}\circ(f_{q})_{2,2})_{\rho:q\to q'\in Q_1}\\
&\in&\bigoplus_{\rho:q\to q'\in Q_1}\Hom_k((FS_2X)_q,(FS_1X')_{q'})\end{eqnarray*}
where $d_{FS_2X,FS_1X'}$ is the map introduced in Section \ref{repsec}. In particular, the induced exact sequence splits. Since $F$ induces a fully faithful embedding when restricted to the quiver with one vertex, there exist morphisms $\varphi_i\in\Hom_{Q(M)}(S_iX,S_iX)$ such that $f_{i,i}=F(\varphi_i)$. Using the functoriality of $\Theta$ for $(g_\rho)_{\rho\in Q(M)_1}:=((\varphi_1)_{m_2}\circ X_{\rho}-X'_{\rho}\circ(\varphi_2)_{m_1})_{\rho\in Q(M)_1}$ we get
\begin{eqnarray*}\Theta^{21}_{X,X'}((g_\rho)_{\rho\in Q(M)_1})&=&
(F(\varphi_{1})_{q'}\circ ((FX)_{\rho})_{1,2}-((FX')_{\rho})_{1,2}\circ F(\varphi_{2})_{q})_{\rho:q\to q'\in Q_1}\\
&=&((f_{q'})_{1,1}\circ ((FX)_{\rho})_{1,2}-((FX')_{\rho})_{1,2}\circ(f_{q})_{2,2})_{\rho:q\to q'\in Q_1}
\\&=&d_{FS_2X,FS_1X'}(f_{1,2})
\end{eqnarray*}
But since $\Theta$ is an isomorphism, this already means that $d_{FS_2X,FS_1X'}(f_{1,2})=0$ and thus $f_{1,2}\in\Hom_Q(FS_2X,FS_1X')$. But since $\Hom_Q(M_2,M_1)=0$, it follows that $f_{1,2}=0$.
\end{proof}
This enables us to prove the main result of this section: 
\begin{thm}\label{thm1}
Let $M=(M_1,\ldots,M_r)$ be an elementary sequence of Schur representations. Then
$F_M$ defined in Section \ref{functorsec} is a fully faithful embedding. In particular, $F_MX$ is indecomposable if and only if $X$ is indecomposable. 

\end{thm}
\begin{proof}
As already mentioned in Remark \ref{bem} it is straightforward that $F=F_M$ is a faithful functor. Thus it remains to show that $F$ is full. To do so let $f:FX\rightarrow FX'$ be a morphism. It induces linear maps $(f_q)_{i,j}:\tilde X_{j,q}\rightarrow \tilde X'_{i,q}$ for all $q\in Q_0$ and all $1\leq i,j\leq r$ such that 
\[\sum_{k=1}^r((FX')_{\rho})_{i,k}\circ(f_{q})_{k,j}=\sum_{k=1}^r(f_{q'})_{i,k}\circ ((FX)_{\rho})_{k,j}\]
for all $\rho\in Q_1$. Thus we get
\begin{eqnarray*}
(e_{i,j})_{\rho}&:=&((FX')_{\rho})_{i,i}\circ(f_{q})_{i,j}-(f_{q'})_{i,j}\circ ((FX)_{\rho})_{j,j}\\&=&\sum_{\substack{k=1\\k\neq j}}^r(f_{q'})_{i,k}\circ ((FX)_{\rho})_{k,j}-\sum_{\substack{k=1\\k\neq i}}^r((FX')_{\rho})_{i,k}\circ(f_{q})_{k,j}\in\Hom_k((FS_jX)_q,(FS_iX')_{q'}).\end{eqnarray*}
But this means that $((e_{i,j})_{\rho})_{\rho\in Q_1}=d_{FS_jX,FS_iX'}(((f_q)_{i,j})_{q\in Q_0})$. In particular, the induced exact sequence $((e_{i,j})_{\rho})_{\rho\in Q_1}\in\Ext(FS_jX,FS_iX')$ splits. Using the second part of Lemma \ref{hom0} there exist isomorphisms $g_i\in \End(FS_iX)=\End(M_i^{\dim X_i})$ and $g_j\in\End(M_j^{\dim X'_j})$ such that $(g_j)_{q'}\circ(e_{i,j})_{\rho}\circ(g_i)^{-1}_q=0$. Indeed, the isomorphism $g$ between the middle term of the exact sequence $E(((e_{i,j})_{\rho})_{\rho\in Q_1})$ and the middle term of the exact sequence $E(0)$ is a block matrix
\[g=\begin{pmatrix}g_j&0\\0&g_i\end{pmatrix}.\]
This already means that we have $(e_{i,j})_{\rho}=0$ for every $\rho\in Q_1$. Now since $\Hom_Q(M_i,M_j)=0$ for $i\neq j$, we have $f_{i,j}=0$ for $i\neq j$. Thus it follows that $f_{i,i}\in\End(FS_iX)$. Since $F$ induces a fully faithful embedding $F|_{M_i}:\Rep(Q(M_i))\rightarrow\mathrm{add}(M_i)$ by the first part of Lemma \ref{hom0}, the claim follows.
\end{proof}
\begin{ex}\label{seqex}
Clearly reduced exceptional sequences defined in Section \ref{experp} are a special case of elementary sequences of exceptional roots. Thus Theorem \ref{thm1} generalizes Theorem \ref{perpcat} (and Corollary \ref{root}). Note that for every exceptional sequence $E$ there exists a corresponding reduced one $E'$ such that $Q(E^{\perp})=Q(E'^{\perp})$.

But we should again point out that there is another generalization: we can consider elementary sequences of Schur representations $M=(M_1,\ldots, M_r)$ rather than exceptional ones.
\end{ex}

\begin{ex}

Maybe the first non-trivial example is obtained when considering the $4$-subspace quiver, see Example \ref{subspace} for a definition. Then we can consider the two exceptional representations $M_{\alpha}$ and $M_{\beta}$ where $\alpha=(1,1,1,0,0)$ and $\beta=(1,0,0,1,1)$. Then we have 
\[Q(M_\alpha,M_\beta)=
\begin{xy}
\xymatrix{m_\alpha\ar@/^/[r]&m_\beta\ar@/^/[l]}\end{xy}
\]
A basis of $\Ext(M_{\alpha},M_{\beta})$ is for instance induced by $\chi^{\alpha\beta}_1:(M_{\alpha})_{q_1}\to (M_{\beta})_{q_0},\,1\mapsto 1$. A basis of $\Ext(M_{\beta},M_{\alpha})$ is for instance induced by $\chi^{\beta\alpha}_2:(M_{\beta})_{q_3}\to (M_{\alpha})_{q_0},\,1\mapsto 1$. Since the indecomposable representations of $Q(M_\alpha,M_\beta)$ have dimension vectors $(n,n+1),\,(n+1,n)$ and $(n,n)$ for $n\in\Nn$, in this way we can construct the real root representations of $Q$ of dimensions $(2n+1,n,n,n+1,n+1)$ and $(2n+1,n+1,n+1,n,n)$ respectively. Moreover, we can construct a $\mathbb{P}^1$-family of indecomposable representations of dimension $(2n,n,n,n,n)$. Note that this family depends on the chosen basis. Moreover, note that if $X$ is a tree module  of $Q(M_\alpha,M_\beta)$, then $FX$ is tree module of $Q$. Moreover, this gives an embedding $\Rep(\tilde{A}_2)\to\Rep(\tilde{D}_4)$ where $\tilde A_2$ and $\tilde D_4$ are considered with the mentioned orientations.
\end{ex}
\begin{ex}

The indecomposable real root representation considered in \cite[Example 4.1]{wei}, which cannot be constructed using Ringel's reflection functors, is also covered by the methods of this section: consider the $8$-subspace quiver and, moreover, the real root $\alpha=(48,1,1,1,15,15,18,18,46)$. 
We denote by $s_i$ the simple root corresponding to $q_i$. Furthermore, we consider the elementary sequence of exceptional representations $M:=(M_{\beta_1},M_{\beta_2}, M_{\beta_3},M_{\beta_4})$ induced by
\[(\beta_1,\beta_2,\beta_3,\beta_4)=(s_0+s_6+s_7+s_8,(2,1,1,1,0,0,2,2,0),s_0+s_4+s_8,s_0+s_5+s_8)^.\]
Then we have $\alpha=16\beta_1+\beta_2+15\beta_3+15\beta_4$ and $Q(M)$ is
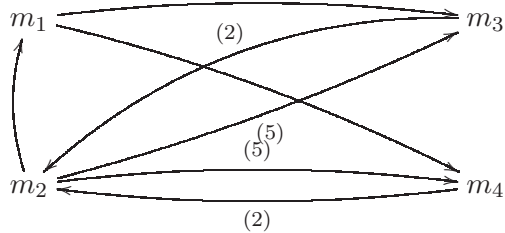
\[
\begin{xy}\xymatrix@R50pt@C150pt{m_1\ar@/^/[r]\ar@/^/[rd]&m_3\ar@/_{20pt}/[ld]_{(2)}\\m_2\ar@/^/[u]\ar@/_{5pt}/[ru]_{(5)}\ar@/^/[r]^{(5)}&m_4\ar@/^/[l]^{(2)}
}\end{xy}
\]
where the number in parentheses indicates the number of arrows. Now the indecomposable representation of dimension $(16,1,15,15)$ yields the indecomposable of dimension $\alpha$.
\end{ex}

\subsection{Extending representations by exceptional representations}\label{s2}
The second condition on sequences of quiver representations, which we want to investigate, yields a generalization of parts of Theorem \ref{ringel}. 
The respective functors are obtained when considering the case of sequences of length two, see also Remark \ref{remrin}. In other words, in comparison to the reflection functor, we can extend a representation by more than one exceptional representation.								

For an additive category $\mathcal C$ we denote by $\mathcal{R}_{\mathcal C}$ its radical. We need the following lemma:
\begin{lem}\label{factor}
Let $X$ and $Y$ be two representations with $\Hom(X,Y)=0$. Then every short exact sequence \[\sesv{X}{Z}{Y}{\psi}{\pi}\]
induces a ring homomorphism $\Phi:\End(Z)\rightarrow\End(X)$ with $f\in\ker(\Phi)$ if and only if $f$ factors over $Y$.
\end{lem}
\begin{proof}Let $f\in\End(Z)$. Since $\Hom(X,Y)=0$, by the universal property of the kernel and cokernel respectively, $f$ induces two unique endomorphisms $g\in\End(X)$ and $h\in\End(Y)$. In particular, we get the following commutative diagram:
\[
\begin{xy}
\xymatrix{0\ar[r]&X\ar[r]^{\psi}\ar[d]^{g}&Z\ar[d]^{f}\ar[r]^{\pi}&Y\ar[d]^{h}\ar[r]&0\\
0\ar[r]&X\ar[r]^{\psi}&Z\ar[r]^{\pi}&Y\ar[r]&0
}\end{xy}
\]
This defines a ring homomorphism $\Phi:\End(Z)\rightarrow\End(X)$. If $g=0$, we have $f\circ \psi=0$. Again by the universal property of the cokernel, we obtain a morphism $\iota:Y\rightarrow Z$ such that $f=\iota\circ \pi$.

The other way around if $f=\iota\circ \pi$ for some $\iota:Y\rightarrow Z$, we obtain $\psi\circ g= f\circ \psi=\iota\circ \pi\circ \psi=0$. Since $\psi$ is injective, we obtain that $g=0$
\end{proof}

Let $M=(M_1,\ldots,M_r)$ be a sequence of indecomposable representations such that the following conditions are satisfied:
\begin{enumerate}\label{conditions}
\item The representations $M_j,\,j\geq 2$ are Schurian.
\item We have $\Hom(M_i,M_j)=\Ext(M_i,M_j)=0$ for $i<j$.
\item If $i,j\neq 1$, we have $\Hom(M_j,M_i)=0$ for $i<j$.
\end{enumerate}
This means that we allow homomorphisms $M_j\rightarrow M_1$ for $j>1$. Moreover, note that the representations $M_i$ do not have to be exceptional. Note that condition $(2)$ assures that $Q(M)$ has no oriented cycles. Now we again consider the functor $F_M$ defined in Section \ref{functorsec}.

Recall that, since $M_i$ is indecomposable, we have that $f\in\Hom(M_i,M_i)$ is in the radical if and only if it is no isomorphism. 
For a representation $X$ of $Q(M)$ we denote by $X_2$ the corresponding representation of the full subquiver $Q(M)\backslash\{m_1\}$. 

Recall that we have a natural isomorphism
$$\Theta^1_X:=(\Theta^{i,1}_{X,X})_{i=2,\ldots,r}:\bigoplus_{i=2}^r\bigoplus_{l=1}^{n_{i1}}\Hom_k(S_iX,S_1X)\cong\bigoplus_{i=2}^r\Ext(M_i^{\dim X_{m_i}},M_1^{\dim X_{m_1}}).$$

With these properties in hand we get the following result:
\begin{thm}\label{indecompconstr}
Let $M$ be a sequence of indecomposable representations satisfying the conditions $(1)-(3)$. If $X$ is a representation of $Q(M)$ such that $\dim X_{m_1}=1$ and such that $\Theta^1_X$ induces an isomorphism 
$$\bigoplus_{i=2}^r\Ext(M_i^{\dim X_{m_i}},M_1)\cong\Ext(FX_2,M_1),$$
 we have that $F_MX$ is indecomposable whenever $X$ is indecomposable. 
\end{thm}
\begin{proof}
Let $X$ be a representation of $Q(M)$. For $F=F_M$ there exists a short exact sequence
\[\sesv{FX_{m_1}}{FX}{FX_2}{\psi}{\pi}.\] Note that $X_{m_1}$ is the simple representation $S_{q_1}$ and we have $FX_{m_1}=M_1$. Since we have $\Hom_Q(M_1,FX_2)=0$, every morphism $f\in\End(FX)$ is of the form
\[f=\begin{pmatrix}f_{11}&f_{12}\\0&f_{22}\end{pmatrix}\]
such that we have $f_{11}\in\Hom_Q(FX_{m_1},FX_{m_1})$, $f_{22}\in\Hom_Q(FX_2,FX_2)$ and, moreover, $f_{12}\in\bigoplus_{q\in Q_0}\Hom_k((FX_2)_q,(FX_{m_1})_q)$.
In particular, we get a commutative diagram
\[
\begin{xy}
\xymatrix{0\ar[r]&FX_{m_1}\ar[r]^{\psi}\ar[d]^{f_{11}}&FX\ar[d]^{f}\ar[r]^{\pi}&FX_2\ar[d]^{f_{22}}\ar[r]&0\\
0\ar[r]&FX_{m_1}\ar[r]^{\psi}&FX\ar[r]^{\pi}&FX_2\ar[r]&0
}\end{xy}
\]
inducing a ring homomorphism $\Phi:\End(FX)\rightarrow\End(FX_{m_1})=\End(M_1)$ by Lemma \ref{factor}. Since $(M_2,\ldots,M_r)$ is an elementary sequence of Schur representations, by Theorem \ref{thm1} the functor $F$ restricted to this sequence is a fully faithful embedding. In particular, there exists a morphism $g_{22}\in\Hom_{Q(M)}(X_2,X_2)$ such that $f_{22}=F(g_{22})$. Moreover, there exists a morphism $g_{11}\in\Hom_{Q(M)}(X_{m_1},X_{m_1})$ and a morphism $\tilde{f}_{11}\in\mathcal{R}_{\mathrm{add}(M_1)}$ such that $f_{11}=F(g_{11})+\tilde{f}_{11}$. Since $f$ is a morphism, we obtain
\begin{eqnarray*}
(d_{FX_2,FX_{m_1}}(f_{12}))_{\rho}&=&(f_{12})_{q'}\circ (FX_2)_{\rho}-(FX_{m_1})_{\rho}\circ (f_{12})_q\\&=&((FX)_{\rho})_{12}\circ (f_{22})_q-(f_{11})_{q'}\circ ((FX_2)_{\rho})_{12}\\
&=&((FX)_{\rho})_{12}\circ (F(g_{22}))_q-(F(g_{11})+\tilde{f}_{11})_{q'}\circ ((FX_2)_{\rho})_{12}
\end{eqnarray*}
for every $\rho:q\rightarrow q'$. If $f_{11}=0$, i.e. $f\in\ker(\Phi)$, we have $g_{11}=0=\tilde{f}_{11}$. Similar to the case considered in Lemma \ref{hom0} and Theorem \ref{thm1} respectively, we can make use of the functoriality of $\Theta^1_X$. Combining it with the fact that $\Theta^1_X$ induces the mentioned isomorphism this first yields 
\[((FX)_{\rho})_{12}\circ (F(g_{22}))_q=0\]
and then $X_{\chi^{i1}_l}\circ (g_{22})_{m_i}=0=(g_{22})_{m_1}\circ X_{\chi^{i1}_l}$ for all $\chi^{i1}_l:m_i\to m_1\in Q(M)_1$ where $i\in\{2,\ldots,r\}$ and $l\in\{1,\ldots,n_{i1}\}$. Note that $(g_{22})_{m_1}=0$ because $g_{22}\in\Hom_{Q(M)}(X_2,X_2)$. In particular, $g_{22}$ extends from a morphism of $X_2$ to a morphism of $X$. 
Since $X$ is indecomposable, every endomorphism of $X$ is either an isomorphism or nilpotent, see \cite[Corollary I.4.8]{ass}. Thus $g_{22}$ is nilpotent because it is not an isomorphism. Thus $F(g_{22})=f_{22}$ is nilpotent and, therefore, also $f$ is nilpotent. In summary we get an embedding 
\[R:=\End(FX)/\ker(\Phi)\hookrightarrow\End(M_1)\]
where $\End(M_1)$ is local. Thus $R$ is local because it is a subring of a local ring. Since $\ker(\Phi)$ only consists of nilpotent elements, $\End(FX)$ is already local. Indeed, obviously every idempotent of $\End(FX)$ is either nilpotent or a unit.
\end{proof}

\begin{rem}\label{theta}
Note that for the implication that $g_{22}$ is a morphism of $X$ (and not only of $X_2$), it is important that $\Theta^1_X$ induces the mentioned isomorphism. This is because otherwise we have 
\[\dim \Ext(FX_2,M_1)<\dim\Ext(\bigoplus_{i=2}^rM_i^{\dim X_{m_i}},M_1)\]
which means that $FX$ might be decomposable.

The condition that $\Theta^1_X$ induces an isomorphism is for instance satisfied if $\Hom(M_i,M_1)=0$ for all $i=2,\ldots,r$ or if we have a decomposition of $\{2,\ldots,r\}$ into disjoint sets $I_1$ and $I_2$ such that $\Hom(M_{i},M_1)=0$ for $i\in I_2$ and $\Ext(M_i,M_j)=0$ for $i\in I_1$ or $j\in I_1$, $i\neq j$ and $i,j\geq 2$. This is clearly satisfied if $r=2$ or if $\Ext(M_i,M_j)=0$ for $i\neq j$ and $i,j\geq 2$.
\end{rem}

\begin{rem}

We obtain an analogous result if we consider sequences $M=(M_1,\ldots,M_r)$ of indecomposable representations such that 
\begin{enumerate}
\item The representations $M_j,\,j\leq r-1$, are Schurian.
\item We have $\Hom(M_i,M_j)=\Ext(M_i,M_j)=0$ for $i<j$.
\item If $i,j\neq r$, we have $\Hom(M_j,M_i)=0$ for $i<j$.
\end{enumerate} 
Thus we allow homomorphisms $M_r\to M_i$ for $i\in\{1,\ldots,r-1\}$. The proof in this case is obtained by a slight modification of the arguments or when considering the opposite quiver obtained when turning around all arrows.
\end{rem}

\begin{rem}\label{remrin}

  Let $X$ be a representation of a quiver $Q$ and $S$ an exceptional representation such that $\Hom(X,S)=0$, $\Hom(S,X)=0$, $\dim\Ext(X,S)=n_1$ and $\dim\Ext(S,X)=n_2$. Then every indecomposable representation of
\[
\begin{xy}
\xymatrix{n_2\ar@/^{0.3cm}/[rr]^{\rho_1}\ar@/_{2mm}/[rr]_{\rho_{n_2}}&\vdots&1
}\end{xy}
\]
with dimension vector as indicated yields an indecomposable representation of dimension  $\underline{\dim} X+n_2\cdot\underline{\dim}S$. In particular, we can construct the indecomposable representation obtained in Theorem \ref{ringel} in this way, i.e. the middle term of the exact sequence
\[\ses{X}{X^S}{\bigoplus_{i=1}^{n_2} S}.\]
Now we have $\Ext(S,X^S)=\Hom(S,X^S)=0$ and since we have $\Hom(X,S)\cong\Ext(X^S,S)$, every indecomposable representation of the quiver 
\[
\begin{xy}
\xymatrix{1\ar@/^{0.3cm}/[rr]^{\rho_1}\ar@/_{2mm}/[rr]_{\rho_{n_1}}&\vdots&n_1
}\end{xy}
\]
gives an indecomposable representation of $Q$ of dimension $\underline{\dim} X^S+n_1\cdot\underline{\dim}S$. In particular, we can construct the representation $X^S_S$ obtained by the reflection functor. 
\end{rem}

Let us consider two examples:
\begin{ex}\label{ex2}
 
First consider the quiver $Q$ given by
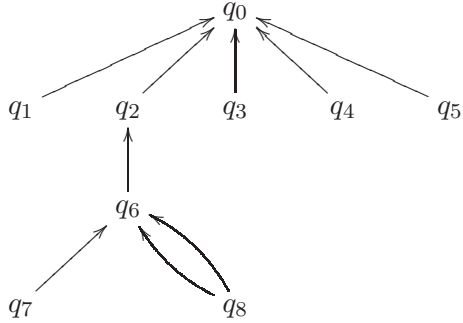
\[
\begin{xy}
\xymatrix{&&q_0\\q_1\ar[rru]&q_2\ar[ru]&q_3\ar[u]&q_4\ar[lu]&q_5\ar[llu]\\&q_6\ar[u]\\q_7\ar[ru]&&q_8\ar@/^{0.2cm}/[lu]\ar@/_{0.2cm}/[lu]}\end{xy}
\]
Moreover, we consider the imaginary root $\alpha_1=(3,1,1,1,1,1,0,0,0)$ and the exceptional roots $\alpha_2=(1,0,0,0,1,1,0,0,0)$, $\alpha_3=(0,0,0,0,0,0,1,1,0)$, $\alpha_4=(0,0,0,0,0,0,1,0,2)$. Choosing an indecomposable representation $M_{\alpha_1}\in{^\perp} M_{\alpha_2}$ such that $\dim\Hom(M_{\alpha_2},M_{\alpha_1})=1$ and defining $M=(M_{\alpha_1},M_{\alpha_2},M_{\alpha_3},M_{\alpha_4})$ the quiver $Q(M)$ is given by 
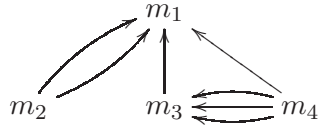
\[
\begin{xy}
\xymatrix@C30pt{&m_1\\m_2\ar@/^{0.2cm}/[ru]\ar@/_{0.2cm}/[ru]&m_3\ar[u]&m_4\ar[lu]\ar[l]\ar@/^{0.2cm}/[l]\ar@/_{0.2cm}/[l]&}\end{xy}
\]
Note that, by Remark $\ref{theta}$ the conditions needed to apply Theorem $\ref{indecompconstr}$ are satisfied. In a sense this example is an application of Ringel's reflection functor and Schofield induction at the same time.
\end{ex}
\begin{ex}
As a second example let us consider the quiver
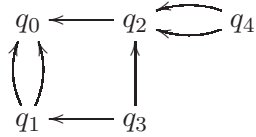
\[
\begin{xy}
\xymatrix{q_0&q_2\ar[l]&q_4\ar@/^{0.2cm}/[l]\ar@/_{0.2cm}/[l]\\q_1\ar@/^{0.2cm}/[u]\ar@/_{0.2cm}/[u]&q_3\ar[l]\ar[u]}\end{xy}
\]
and the root $\alpha_1=(2,2,0,0,0)$ and the exceptional roots $\alpha_2=(0,0,2,0,1)$ and $\alpha_3=(0,0,0,1,0)$. Note that $\alpha_1$ corresponds to an imaginary non-Schur root of the Kronecker quiver. Choosing three indecomposable representations of the respective dimensions we obtain the quiver $Q(M)$ given by
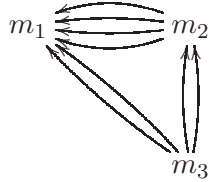
\[
\begin{xy}
\xymatrix@C40pt@R40pt{m_1&m_2\ar@/^{0.1cm}/[l]\ar@/_{0.1cm}/[l]\ar@/^{0.3cm}/[l]\ar@/_{0.3cm}/[l]\\&m_3\ar@/^{0.1cm}/[lu]\ar@/_{0.1cm}/[lu]\ar@/^{0.1cm}/[u]\ar@/_{0.1cm}/[u]}\end{xy}
\]
In both cases every root $\beta$ with $\beta_{m_1}=1$ and an indecomposable representation of this dimension gives rise to an indecomposable representation of $Q$.
\end{ex}

\section{A remark on the decomposition of roots}
\noindent\setcounter{section}{1}\setcounter{thm}{1}We fix a quiver $Q$ with vertices $\{1,\ldots,n\}$.  The following question was posed in the introduction:
\begin{que}
For which (non-Schurian) roots of a fixed quiver $Q$ there exists at least one indecomposable representation which can be constructed by Theorems \ref{thm1} and \ref{indecompconstr} (which include Ringel's reflection functor and Schofield induction)?
\end{que}
\setcounter{section}{4}\setcounter{thm}{0}
This question is closely related to the question of decomposing a root into smaller ones which is dealt with in this section. For simplicity we restrict to the case where the smaller roots are exceptional, i.e. we deal with the following question:

\begin{que} Fix a root $\alpha$ of a (tree-shaped) quiver. Does there exist a non-trivial reduced exceptional sequence (resp. elementary sequence of exceptional roots)  such that $\alpha$  is a positive linear combination of the roots contained in the sequence?
\end{que} 
In this question, by trivial decomposition we mean the one induced by the trivial exceptional sequence $(S_1,\ldots, S_n)$, i.e. $\alpha=\sum_{i=1}^n\alpha_i\cdot\underline{\dim} S_i$.  By Corollary \ref{root} the existence of such decompositions for tree-shaped quivers together with Proposition \ref{dec} would imply the existence of indecomposable tree modules for every root. Indeed, we could start an induction on the set of roots of all quivers. Note that for the generalized Kronecker quiver we only have the trivial decomposition. But even for a quiver with three vertices it is not clear which roots have such a decomposition. 

An algorithm to obtain a possibly trivial reduced exceptional sequence for a root, which is not a Schur root, is given as follows (recall that Schur roots are covered by \cite[Section 3]{wei}): we start with the canonical decomposition of $\alpha$, say $\alpha=\bigoplus_{i=1}^n\alpha_i^{k_i}$, see \cite{sch} and \cite[Section 4]{dw} for an algorithm to compute it and some additional properties. Then either $\alpha_1$ or $\alpha_n$ is exceptional. We also know that $\beta:=\sum_{i=2}^n\alpha_i^{k_i}\in\alpha_1^{\perp}$. Thus if $\alpha_1$ is exceptional, there exists a decomposition $\beta=\sum_{i=1}^t\beta_i^{l_i}$ into exceptional roots $\beta_1,\ldots,\beta_t$ such that $(\beta_1,\ldots,\beta_t)$ is a reduced exceptional sequence. Such a decomposition exists because we can always take the simple roots in $\alpha_1^{\perp}$. In particular, $(\alpha_1,\beta_1,\ldots,\beta_t)$ is an exceptional sequence. Note that we have $\ext(\beta,\alpha_1)=0$ because we started with the canonical decomposition. If there exists a $\beta_i$ such that $\hom(\beta_i,\alpha_1)\neq 0$, i.e. the sequence is not satisfying the requested properties, we can proceed in the same manner and decompose $\gamma:=\alpha_1^{k_1}+\sum_{i=1}^{t-1}\beta_{i}^{l_i}$ into exceptional roots of $^\perp\beta_t$. We proceed in the same manner until we obtain a reduced exceptional sequence. This algorithm terminates because the roots are getting smaller in every step. Thus we are forced to end up with the trivial exceptional sequence if we do not obtain a reduced exceptional sequence before.

Unfortunately, it seems that there does not always exist a non-trivial decomposition as Example \ref{ex1} suggests. At least the algorithm from above is not applicable to obtain a non-trivial sequence in that case. The next question would be if there is a decomposition induced by an elementary sequence of exceptional roots for every root. But this seems to be an even harder problem.

\begin{ex}\label{ex1}

Let $Q$ be the $5$-subspace quiver. Then the canonical decomposition of the isotropic root $\alpha=(10,3,3,3,3,8)$ is given by
\[\alpha=(6,2,2,2,2,4)\oplus(1,1,0,0,0,1)\oplus\ldots\oplus(1,0,0,0,1,1).\]
In the first step we get an exceptional sequence
\[((3,1,1,1,1,0),S_5,(1,1,0,0,0,1),\ldots,(1,0,0,0,1,1)).\] 
In the next step we obtain the sequence
\[((3,1,1,1,1,0),(1,1,0,0,0,0),\ldots,(1,0,0,0,1,0),S_5).\] 
Now it easy to check that we end up with the trivial decomposition. 

But there exists a decomposition induced by an elementary sequence of exceptional roots which is
\[\alpha=(1,1,1,0,0,0)+(1,0,0,1,1,0)+2\cdot(1,1,0,0,0,1)+\ldots+2\cdot(1,0,0,0,1,1).\]
Thus we can construct indecomposables of dimension $\alpha$ when constructing indecomposables of the quiver
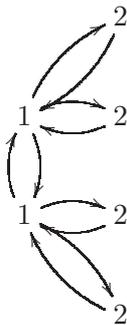
\[
\begin{xy}\xymatrix{&2\ar@/^/[ld]\\1\ar@/^/[ru]\ar@/^/[d]\ar@/^/[r]&2\ar@/^/[l]\\1\ar@/^/[rd]\ar@/^/[u]\ar@/^/[r]&2\ar@/^/[l]\\&2\ar@/^/[lu]
}\end{xy}
\]
with dimension vector as indicated.

We should mention that we can also construct it by successively applying Ringel's reflection functor and thus with the functor of Section \ref{s2}.
\end{ex}
\begin{ex}

But there are plenty of (even non-Schurian) roots for which the algorithm gives a non-trivial decomposition. For instance if $\alpha=(3,2,2,1,1)$ and $Q$ is the $4$-subspace quiver, starting with the canonical decomposition $(2,1,1,1,1)\oplus(1,1,1,0,0)$ the algorithm terminates with the non-trivial reduced exceptional sequence
\[((1,0,1,0,0),(1,0,0,1,1),S_1).\]
\end{ex}
\section{Taking self-extensions into account }\label{5}
\noindent 
It is a natural to ask whether it is possible to take self-extensions of non-exceptional representations into account when considering the functor of Section \ref{functorsec}. This should translate to adding $n$ loops to a vertex which corresponds to a Schur representation $M$ with $\dim\Ext(M,M)=n$. Actually, we will see that there is a natural way to generalize the functor. But we will state a counterexample which shows that this functor does not preserve indecomposability even in the case of a quiver with only one vertex and loops. In particular, the functor is not full.

By $L(n)$ we denote the quiver having one vertex denoted by $m$ and $n$ loops. Let $M$ be a representation of $Q$ with $\End(M)=k$ and $\dim\Ext(M,M)=n$. Fix a tree-shaped basis $\{b_1,\ldots,b_n\}$ of $\Ext(M,M)$. Let $X=(X_m,(X_l)_{l=1,\ldots,n})$ be a representation of  $L(n)$ and define $\tilde X_{q}:=M_q\otimes_k X_{m}$ for $q\in Q_0$. Define a representation $FX$ of $Q$ by the vector spaces
\[(FX)_q= \tilde X_{q}\text{ for all }q\in Q_0\]
and for $\rho:q\rightarrow q'$ we define linear maps $(FX)_{\rho}=\tilde X_q\to \tilde X_{q'}$ by
\[(FX)_{\rho}=M_{\rho}\otimes_k\mathrm{id}_{X_{m}}+\sum_{l=1}^{n}(b_l)_{\rho}\otimes_kX_l:\tilde X_{q}\rightarrow \tilde X_{q'}.\]

Let $f:X\rightarrow X'$ be a morphism. Then we define $Ff:FX\rightarrow FX'$ by
\[(Ff)_q:=\mathrm{id}_{M_q}\otimes_kf_{m}:\tilde X_{q}\rightarrow  \tilde X'_{q'}.\]
Note that it is easy to check that $F$ is faithful. Moreover, it indeed defines a functor because for every morphism $f:X\to X'$, we have
\begin{eqnarray*}
(FX')_{\rho}\circ (Ff)_{q}&=&(M_{\rho}\otimes\mathrm{id}_{X'_m}+\sum_{l=1}^n(b_l)_{\rho}\otimes X'_l)\circ\mathrm{id}_{M_q}\otimes f_m\\
&=&M_{\rho}\otimes f_m+\sum_{l=1}^n(b_l)_{\rho}\otimes X'_l\circ f_m\\
&=&M_{\rho}\otimes f_m+\sum_{l=1}^n(b_l)_{\rho}\otimes f_m\circ X_l\\
&=&\mathrm{id}_{M_{q'}}\otimes f_m\circ(M_{\rho}\otimes\mathrm{id}_{X_m}+\sum_{l=1}^n(b_l)_{\rho}\otimes X_l)\\&=&(Ff)_{q'}\circ (FX)_{\rho}.
\end{eqnarray*}
Now it would suffice to have that $f$ is also full in order to have that $F$ preserves indecomposability. But the following example shows that in general this is not the case, even for Schurian representation (resp. for stable representations which we do not define in this paper). Consider the generalized Kronecker quiver $K(3)$ with two vertices $\{q,q'\}$ and three arrows $\{a,b,c\}$ from $q$ to $q'$. Moreover, let $M$ be the representation of dimension $(2,3)$ defined by the matrices
\[M_a=\begin{pmatrix}0&0\\0&0\\0&1\end{pmatrix},\,M_b=\begin{pmatrix}1&0\\0&1\\0&0\end{pmatrix},\,M_c=\begin{pmatrix}0&0\\1&0\\0&0\end{pmatrix}.\]
The corresponding coefficient quiver is given by
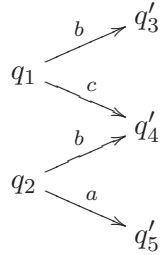
\[
\begin{xy}
\xymatrix@R5pt@C30pt{&q'_3\\q_1\ar[ru]^b\ar[rd]^c&\\&q'_4\\q_2\ar[ru]^b\ar[rd]^a&\\&q'_5}
\end{xy}\]
where the indices of the arrows name the original arrows. 

Now it is easy to check that $M$ is even stable with respect to the standard stability defined by the linear form $\Theta=(1,0):\Zn K(3)_0\to\Zn$, where we refer to \cite{kin} for more details on stable representations. In particular, $M$ is Schurian. Clearly,
every arrow of the set $\{\rho:s\to t\mid \rho\in K(m)_1,\,s\in\{q_1,q_2\},\,t\in\{q'_1,q'_2,q'_3\}\}$ defines an element of $\bigoplus_{\rho:q\rightarrow q'\in Q_1}\Hom_k(M_q,M_{q'})$ in the natural way. Using this notation, a tree-shaped bases of $\Ext(M,M)$ is given by the arrows
\[q_1\xlongrightarrow{a}q'_3,\,\,q_1\xlongrightarrow{a}q'_4,\,\,q_2\xlongrightarrow{b}q'_3,\,\,q_2\xlongrightarrow{b}q'_5,\,\,q_2\xlongrightarrow{c}q'_3,\,\,q_2\xlongrightarrow{c}q'_5\]
Thus the representation $M'$ defined by the matrices
\[M'_a=\begin{pmatrix}1&0\\1&0\\0&1\end{pmatrix},\,M'_b=\begin{pmatrix}1&0\\0&1\\0&1\end{pmatrix},\,M'_c=\begin{pmatrix}0&1\\1&0\\0&1\end{pmatrix}\]
can be constructed using the functor from above. But since 
\[g_q=\begin{pmatrix}0&1\\0&1\end{pmatrix},\,\,g_{q'}=\begin{pmatrix}0&0&1\\0&0&1\\0&0&1\end{pmatrix}\]
is a non-trivial idempotent of $\End(M')$, the representation $M'$ is not indecomposable.

\end{document}